\documentclass[a4paper,12pt,dvipdfm]{amsart}
\usepackage{amsmath}
\usepackage{amssymb}
\usepackage{enumerate}
\usepackage{amscd}
\usepackage{graphics}
\usepackage{latexsym}
\usepackage{verbatim} 
\usepackage{url}

\theoremstyle{plain}
\newtheorem{thm}{{\bf Theorem}}[section]
\newtheorem{cor}[thm]{{\bf  Corollary}}
\newtheorem{prop}[thm]{{\bf Proposition}}
\newtheorem{lemma}[thm]{{\bf Lemma}}
\newtheorem{fact}[thm]{{\bf Fact}}
\newtheorem{claim}[thm]{{\bf Claim}}

\theoremstyle{definition}
\newtheorem{define}[thm]{{\bf Definition}}

\newtheorem{question}[thm]{{\bf Question}}

\newcommand{\cf}{\mathord{\mathrm{cf}}}

\newcommand{\dom}{\mathord{\mathrm{dom}}}
\newcommand{\size}[1]{\left\vert {#1} \right\vert}
\newcommand{\p}{\mathcal{P}}

\newcommand{\col}{\mathord{\mathrm{Col}}}

\newcommand{\seq}[1]{\langle {#1} \rangle}

\newcommand{\ka}{\kappa}
\newcommand{\la}{\lambda}
\newcommand{\om}{\omega}

\newcommand{\bbC}{\mathbb{C}}

\newcommand{\bbP}{\mathbb{P}}
\newcommand{\bbQ}{\mathbb{Q}}
\newcommand{\bbR}{\mathbb{R}}

\newcommand{\ZFC}{\mathsf{ZFC}}

\title[Generically extendible cardinals]{Generically extendible cardinals}
\author[T. Usuba]{Toshimichi Usuba}
\address[T. Usuba]
{Faculty of Science and Engineering,
Waseda University, 
Okubo 3-4-1, Shinjyuku, Tokyo, 169-8555 Japan}
\email{usuba@waseda.jp}
\begin{document}

\begin{abstract}
In this paper, 
we study the notion of a generically extendible cardinal, which is a generic version of an extendible cardinal. We prove that the generic extendibility of $\om_1$ or $\om_2$ has 
small consistency strength, but that of a cardinal $>\om_2$ does not.
We also consider some results concerned with generically extendible cardinals, such as
indestructibility, generic absoluteness of the reals, and Boolean valued second order logic.
\end{abstract}
\maketitle

\section{Introduction}
A \emph{large cardinal} is usually defined by
the existence of an elementary embedding from some transitive (set or class) model $M$
into a transitive (set or class) model $N$.
A \emph{generic large cardinal} has a similar definition, but it is characterized by the existence of a \emph{generic  elementary embedding}:
An elementary embedding $j:M \to N$ where $M$ is a model lies in $V$, but
$j$ and $N$ are living in a certain generic extension.
The following are typical examples of generic large cardinals:
\begin{enumerate}
\item A cardinal $\ka$ is \emph{generically measurable}
if there is a poset $\bbP$ which forces that
there is a transitive model $N$ and an elementary embedding $j:V \to N$
with critical point $\ka$.
\item A cardinal $\ka$ is \emph{generically strong}
if for every $\alpha>\ka$, there is a poset $\bbP$ which forces that
there is a transitive model $N$ with $V_\alpha^{V^\bbP} \subseteq N$ and an elementary embedding $j:V \to N$
with critical point $\ka$ and $\alpha<j(\ka)$.
\end{enumerate}
In contrast with  usual large cardinals,
generic large cardinals could be small, e.g.,
it is consistent that 
$\om_1$ is generically measurable.

Recall that, a cardinal $\ka$ is \emph{extendible}
if for every $\alpha>\ka$,
there is $\beta$ and an elementary embedding $j:V_\alpha \to V_\beta$
with critical point $\ka$ and $\alpha<j(\ka)$.
Extendible cardinal is known as one of very strong large cardinals.
In the spirit of generic large cardinals,
we can naturally define the notion of \emph{generically extendible cardinal}
in the following manner.
\begin{define}
An uncountable cardinal $\ka$ is \emph{generically extendible}
if for every $\alpha>\ka$,
there is a poset $\bbP$ which forces the following:
There is  $\beta>\alpha$ and an elementary embedding $j:V_\alpha \to V_\beta^{V^\bbP}$
with critical point $\ka$ and $\alpha<j(\ka)$.
\end{define}
This notion is introduced by
Ikegami \cite{IV2} for the study of
\emph{the Boolean valued second order logic} (see Ikegami-V\"a\"an\"anen \cite{IV}
and Ikegami \cite{IV2}).

The consistency strength of known generic large cardinals is close to the original ones:
For example, the consistency of the existence of a generic measurable cardinal
is equivalent to a measurable cardinal.
However, our generic extendible cardinal is drastically weaker than
an extendible cardinal; Suppose there is a proper class of Woodin cardinals.
Then the stationary tower forcing (see Larson \cite{L}) exemplifies
that \emph{every} regular uncountable cardinal is generically extendible.
By this observation, Ikegami asked the following question:
\begin{question}
How large is the consistency strength of the existence of a generically extendible cardinal?
Does the existence of a generically extendible cardinal imply $0^{\#}$?
\end{question}

The main result of this paper is an answer to this question.
We show that 
the consistency strength of the existence of a generically extendible cardinal is not so strong.
In fact, it is equiconsistent with \emph{virtually extendible cardinal}
introduced in Bagaria-Gitman-Schindler \cite{BGS}
and Gitman-Schindler \cite{GS}

\begin{define}[Bagaria-Gitman-Schindler \cite{BGS}, Gitman-Schindler \cite{GS}]
An uncountable cardinal $\ka$ is \emph{virtually extendible}
if for every $\alpha>\ka$,
there is a poset $\bbP$ which forces the following:
There is $\beta$ and an elementary embedding $j:V_\alpha \to V_\beta$
with critical point $\ka$ and $\alpha<j(\ka)$.
\end{define}
Note that, in the definition of virtually extendible,
an elementary embedding $j:V_\alpha \to V_\beta$  is living in a generic extension,
but the domain $V_\alpha$ and the target $V_\beta$ are in $V$.

In \cite{BGS} and \cite{GS},
they proved that if $\ka$ is virtually extendible, then so is in $L$.
Hence the existence of a virtually extendible cardinal 
is consistent with $V=L$, in particular 
it does not imply $0^{\#}$.
Moreover, they proved that 
every Silver indiscernible is virtually extendible in $L$.

The following theorem shows that the consistency strength of the generic extendibility of $\om_1$ or $\om_2$
is not so strong, and it does not imply $0^{\#}$.
\begin{thm}\label{thm1}
The following theories are equiconsistent:
\begin{enumerate}
\item $\ZFC$+``there exists a virtually extendible cardinal''.
\item $\ZFC$+``there exists a generically extendible cardinal''.
\item $\ZFC$+``$\om_1$ is generically extendible''.
\item $\ZFC$+``$\om_2$ is generically extendible''.
\end{enumerate}
\end{thm}
On the other hand,
the existence of a generically extendible cardinal $>\om_2$
is strong as implying $0^{\#}$.

\begin{thm}\label{thm2}
If there exists a  generically extendible cardinal $>\om_2$,
then $0^{\#}$ exists.
\end{thm}

The last half of this paper is devoted to the study of various topics which are relevant to generically extendible cardinals.
In Section 5, we consider the indestructibility of generically extendible cardinals by forcing.
Laver \cite{L} proved that, after some preparation forcing,
the supercompactness of $\ka$ is indestructible by $<\ka$-directed closed forcing.
It is also known that a preparation forcing is necessary;
After adding one Cohen real, any $<\ka$-closed forcing
which adds a new subset of $\ka$ must destroy the supercompactness of $\ka$ (Hamkins \cite{H2}).
Unlike a supercompact cardinal,
an extendible cardinal is never indestructible:
Every non-trivial $<\ka$-closed forcing must destroy the extendibility of $\ka$
(Bagaria-Hamkins-Tsaprounis-Usuba \cite{BHTU}).
For our generically extendible cardinal,
we prove the following indestructibility property without any preparation forcing.

\begin{thm}
Let $\ka$ be a generically extendible cardinal.
\begin{enumerate}
\item Every $\ka$-c.c. forcing preserves the generic extendibility of $\ka$.
\item If $\ka=\om_1$, every proper forcing  preserves the generic extendibility of $\ka$.
\end{enumerate}
\end{thm}

In Section 6, we consider  generic absoluteness of the reals.
For a class $\Gamma$ of posets, 
let us say that \emph{${\bf \Sigma}_n^1$-absoluteness for $\Gamma$ holds}
if  for every real $r$ and a $\Sigma_n^1$ or $\Pi_n^1$-formula $\varphi(x)$ of the reals,
if $\varphi(r)$ holds in $V$ then so does in $V^{\bbP}$ for every $\bbP \in \Gamma$.
We prove that if $\om_1$ is generically extendible, then
${\bf \Sigma}^1_3$-absoluteness holds, 
and ${\bf\Sigma}^1_4$-absoluteness for proper forcing holds too.
However full ${\bf\Sigma}^1_4$-absoluteness could not.
\begin{thm}
Suppose $\om_1$ is generically extendible.
\begin{enumerate}
\item  ${\bf \Sigma}^1_3$-absoluteness  holds.
\item ${\bf \Sigma}^1_4$-absoluteness for proper forcing holds.
\end{enumerate}
\end{thm}
\begin{thm}
It is consistent that
$\om_1$ is generically extendible,
and ${\bf \Sigma}^1_4$-absoluteness for $\om_1$-preserving forcing fails.
\end{thm}
Using this theorem, we construct a model in which
$\om_1$ is generically extendible and there is an $\om_1$-preserving forcing
which destroys the generic extendibility of $\om_1$.

In Section 7, we study  Boolean valued second order logic.
We show that the Hanf number of Boolean valued second order logic
is strictly smaller than the least generically extendible,
hence the number could be small as $\om$ or $\om_1$.

Here we present some notations which may not be general.
For an ordinal $\alpha$, let $V_\alpha$ be the set of all sets with rank $<\alpha$.
For a transitive model $M$ of (a large fragment of) $\ZFC$,
$V_\alpha^M$ is the relativization of $V_\alpha$ to $M$,
that is, the set of all sets in $M$ with rank $<\alpha$.

Let $X$ be a set.
For a poset $\bbP$ and a $(V, \bbP)$-generic $ G$,
let $X[G]=\{\sigma_G \mid \sigma \in X$ is a $\bbP$-name$\}$,
where $\sigma_G$ is the interpretation of $\sigma$ by $G$.
It is known that if $\alpha$ is a limit ordinal with $\bbP \in V_\alpha$,
then $V_\alpha[G]=V_\alpha^{V[G]}$.

For a poset $\bbP$, if we do not need to specify a $(V, \bbP)$-generic filter,
we denote a generic extension via $\bbP$ by $V^\bbP$.

For a regular cardinal $\la$ and a set $X$,
let $\col(\la, X)$ be the standard $\la$-closed poset which adds a surjection from $\la$ to $X$.
When $\la<\ka$ are regular cardinals with $\alpha^{<\la}<\ka$ for every $\alpha<\ka$,
let $\col(\la,<\ka)$ denote the standard $\la$-closed $\ka$-c.c.~poset which 
forces $\ka=\la^+$.
 
\section{Generic extendibility of $\om_1$}

In this section, we establish the consistency of the generic extendibility of $\om_1$.
First we present some known facts about virtually extendible cardinals.

\begin{fact}[Theorem 3.8 in \cite{BGS}, Theorems 4.6 and 4.7 in \cite{GS}]
Suppose $\ka$ is virtually extendible.
\begin{enumerate}
\item $\ka$ is inaccessible.
\item $\ka$ is virtually extendible in $L$.
\end{enumerate}
\end{fact}

The following lemma is immediate from the definition of virtually extendible cardinals.
\begin{lemma}
If there is a virtually extendible cardinal,
then there are proper class many inaccessible cardinals.
\end{lemma}
\begin{proof}
Fix a large $\alpha>\ka$.
In some generic extension,
there is $\beta>\alpha$ and an elementary embedding $j:V_\alpha \to V_\beta$
with critical point $\ka$ and $\alpha<j(\ka)$.
Since $\ka$ is inaccessible in $V$, we have that
$j(\ka)$ is inaccessible in $V_\beta$, and so is in $V$.
\end{proof}

An \emph{elementary embedding} means an elementary embedding
between $\in$-structures with the language of set theory.

\begin{fact}[Proposition 2.7 in \cite{BGS}, 
Corollary 3.2 in \cite{GS}]\label{GS thm}
Let $M$ be a transitive (set or class) model of a sufficiently large fragment of $\ZFC$.
Let $X, Y \in M$ be transitive sets, and $a \subseteq X$ a finite set.
Suppose there is an elementary embedding $j:X \to Y$ ($j$ may not be in $M$).
Then the forcing $\col(\om, \size{X}^M)$ over $M$ adds an elementary 
embedding $i:X \to Y$ such that $\mathrm{crit}(i)=\mathrm{crit}(j)$ (if $j$ has the critical point) and
$i(x)=j(x)$ for every $x \in a$.
\end{fact}
For the reader's convenience,
we present a proof of this fact.
\begin{proof}
Take a $(V, \col(\om, \size{X}^M))$-generic $G$,
and fix $\seq{x_n \mid n<\om} \in M[G]$ an enumeration of $X$.
Let $T$ be the set of all finite partial elementary embeddings $i$ from $X$ to $Y$ such that:
\begin{enumerate}
\item $\dom(i)=\{x_k \mid k<n\}$ for some $n<\om$.
\item The critical point of $i$ is the same to $j$, that is,
$i(\alpha)=\alpha$ for $\alpha \in \dom(i) \cap \mathrm{crit}(j)$, and if $\mathrm{crit}(j) \in \dom(i)$ then $i(\mathrm{crit}(j))>\mathrm{crit}(j)$.
\item $i(x)=j(x)$ for every $x \in \dom(i) \cap a$.
\end{enumerate}
The set $T$ with the inclusion forms a tree of height $\om$ and $T \in M[G]$.

The set $\{j \restriction \{x_k \mid k<n\} \mid n<\om\}$ is a cofinal branch of $T$,
hence $T$ is ill-founded in $V[G]$.
Since the ill-foundedness of a tree is absolute between $V[G]$ and $M[G]$,
$T$ is ill-founded in $M[G]$.
If $B \in M[G]$ is a cofinal branch of $T$,
then $\bigcup B$ is a required elementary embedding from $X$ to $Y$.
\end{proof}

As a corollary of this fact, we have:
\begin{lemma}
Suppose $\ka$ is virtually extendible.
Then for every $\alpha>\ka$,
$\col(\om, \size{V_\alpha})$ forces that
``there is $\beta$ and an elementary embedding $j:V_\alpha \to V_\beta$
with critical point $\ka$ and $\alpha<j(\ka)$''.
\end{lemma}
These facts yield the following.
\begin{prop}\label{2.2}
If $\ka$ is generically extendible 
then $\ka$ is virtually extendible in $L$.
\end{prop}
\begin{proof}
To show that $\ka$ is virtually extendible in $L$,
fix a large $\alpha>\ka$ such that $V_\alpha$ is a model of  a sufficiently large fragment of $\ZFC$.
Since $\ka$ is generically extendible,
there is a poset $\bbP$ which forces:
There is $\beta$ and an elementary embedding $j:V_\alpha \to V_\beta^{V^\bbP}$ with
critical point $\ka$ and $\alpha<j(\ka)$.
The restriction $j \restriction V_\alpha^L$ is an elementary embedding from
$V_\alpha^L$ to $V_\beta^L$.
By Fact \ref{GS thm},
the forcing $\col(\om, \size{V_\alpha^L})$ over $L$ adds
an elementary embedding $i:V_\alpha^L \to V_\beta^L$ with critical point $\ka$
and $\alpha<i(\ka)$.
Thus $\ka$ is virtually extendible in $L$.
\end{proof}

Combining this proposition with the next one,
we have the equiconsistencies of  (1), (2), and (3) in Theorem \ref{thm1}.

\begin{prop}\label{2.5}
Suppose $\ka$ is virtually extendible.
Then $\col(\om, <\ka)$ forces that ``$\om_1$ is generically extendible''.
\end{prop}
\begin{proof}
Take a $(V, \col(\om, <\ka))$-generic $G$.
We show that, in $V[G]$, $\ka$ is generically extendible.
Fix an arbitrary large inaccessible $\alpha>\ka$.
In $V$, we can fix $\beta$ such that
$\Vdash_{\col(\om, \size{V_\alpha})}$``there is an elementary embedding $j:V_\alpha \to V_\beta$
with critical point $\ka$ and $\alpha<j(\ka)$''.
We may assume that there is $\delta$ such that 
$\Vdash_{\col(\om, \size{V_\alpha})}$``$\delta=j(\ka)$''. 
Note that $\size{V_\alpha}=\alpha<\delta<\beta$.

Take a $(V, \col(\om, <\delta))$-generic $H$ with $G \in V[H]$.
We can assume that $G=H \cap \col(\om, <\delta)$. 
Note that $V[H]$ is a generic extension of $V[G]$.
Since $\delta>\size{V_\alpha}$, 
in $V[H]$ we can construct a $(V, \col(\size{V_\alpha}))$-generic filter.
Hence we can find an elementary embedding $j:V_\alpha \to V_\beta$ with
critical point $\ka$ and $j(\ka)=\delta$.
In $V[H]$, $j$ can be extended to $j:V_\alpha[G] \to V_\beta[H]$ in the canonical way.
We know $V^{V[G]}_\alpha=V_\alpha[G]$ and $V^{V[H]}_\beta=V_\beta[H]$.
Hence, in $V[H]$, $j$ is an elementary embedding from $V^{V[G]}_\alpha$ to $V^{V[H]}_\beta$
with critical point $\ka$ and  $\alpha<j(\ka)$, as required.
\end{proof}

\section{generic extendibility of $\om_2$}
In this section, we show that if there is a virtually extendible  cardinal,
then there is a forcing extension in which $\om_2$ is generically extendible.
This result immediately yields 
the equiconsistency of  (1) and (4) in Theorem \ref{thm1}.

In \cite{Laver}, Laver proved that every supercompact cardinal admits a \emph{Laver diamond},
that is, if $\ka$ is supercompact,
then there is a function $d:\ka \to V_\ka$ such that
for every set $x$ and  cardinal $\la \ge \size{\mathrm{trcl}(x)}$,
there is a $\la$-supercompact embedding $j:V \to M$ with
critical point $\ka$ and $j(d)(\ka)=x$.
We prove that a virtual extendible cardinal also admits a Laver diamond.
\begin{lemma}\label{weak laver diamond}
Suppose $\ka$ is virtually extendible.
Then there is a function $d^*:\ka \to V_\ka$
such that for every set $x$ with rank $> \ka$,
$\col(\om, \size{V_{\mathrm{rank}(x)}})$ forces
that ``there is $\beta>\mathrm{rank}(x)$ and
an elementary embedding $j:V_{\mathrm{rank}(x)} \to V_{\beta}$
with critical point $\ka$, $\mathrm{rank}(x)<j(\ka)$,
and $j(d^*)(\ka)=x$''.
\end{lemma}
\begin{proof}
First note the following claim, which is immediate from
the virtual extendibility of $\ka$.
\begin{claim}
For every $\alpha>\ka$, there is an inaccessible cardinal $\theta>\alpha$
with $V_\ka \prec V_{\theta}$.
\end{claim}

Now we define $d^* \restriction \alpha$ for $\alpha<\ka$ by induction.
Let $\alpha<\ka$,
and suppose $d^* \restriction \alpha$ is defined.
Let $\alpha$ be an inaccessible cardinal,
and suppose the following condition $(*)$ holds for $\alpha$ in $V_\ka$:
\begin{quote}
There is a set $y$ with $\mathrm{rank}(y) > \alpha$
such that 
$\col(\om, \size{V_{\mathrm{rank}(y)}})$ forces
that ``there is no $\beta>\mathrm{rank}(y)$ and
an elementary embedding $i:V_{\mathrm{rank}(y)} \to V_{\beta}$
with the critical point $\alpha$, $\mathrm{rank}(y)<i(\ka)$,
and $i(d^*\restriction \alpha)(\alpha)=y$''.
\end{quote}
Fix such a $y_\alpha \in V_\ka$ with minimal rank,
and let $d^*(\alpha)=y_\alpha$.
If $\alpha$ is not inaccessible or $(*)$ fails, then put $d^*(\alpha)=0$.

Now we prove that $d^*$ is as required.
If not, pick a set $x$ such that 
$\col(\om, \size{V_{\mathrm{rank}(x)}})$ forces
that ``there is no $\beta>\mathrm{rank}(x)$ and
an elementary embedding $j:V_{\mathrm{rank}(x)} \to V_{\beta}$
such that the critical of $j$ is $\ka$, $\mathrm{rank}(x)<j(\ka)$,
and $j(d^*)(\ka)=x$''.
Fix a large inaccessible cardinal $\theta$ with $x \in V_\theta$ and
$V_\ka \prec V_{\theta}$.

Take a $(V, \col(\om, \size{V_\theta}))$-generic $G$ and we work in $V[G]$.
Since $\ka$ is virtually extendible in $V$,
there is an elementary embedding $j:V_\theta \to V_\gamma$ for some $\gamma$
such that the critical point of $j$ is $\ka$ and $\theta<j(\ka)$.
$\gamma$ is inaccessible in $V$ and
$V_{j(\ka)} \prec V_{\gamma}$.
Note that $j(d^*) \restriction \ka=d^*$.
\begin{claim}
In $V_{j(\ka)}$,
$\col(\om, \size{V_{\mathrm{rank}(x)}})$ forces
that ``there is no $\beta>\mathrm{rank}(x)$ and
an elementary embedding $i:V_{\mathrm{rank}(x)} \to V_{\beta}$
such that the critical point of $i$ is $\ka$, $\mathrm{rank}(x)<i(\ka)$,
and $i(d^*)(\ka)=x$''.
\end{claim}
\begin{proof}
If not, then in $V$,
$\col(\om, \size{V_{\mathrm{rank}(x)}})$  forces
that ``there is  $\beta>\mathrm{rank}(x)$ and
an elementary embedding $i:V_{\mathrm{rank}(x)} \to V_{\beta}$
such that the critical of $i$ is $\ka$, $\mathrm{rank}(x)<i(\ka)$,
and $i(d^*)(\ka)=x$''.
This contradicts to the choice of $x$.
\end{proof}
Hence, in $V_{j(\ka)}$, the condition $(*)$ holds for $\ka$.
Let $j(d^*)(\ka)=x^* \in V_{j(\ka)}$.
By the definition of $d^*$, the rank of $x^*$ is smaller or equal to $x$,
and, in $V_{j(\ka)}$, $\col(\om, \size{V_{\mathrm{rank}(x^*)}})$ forces
that:
\begin{quote}
There is no $\beta>\mathrm{rank}(x^*)$ and
an elementary embedding $i:V_{\mathrm{rank}(x^*)} \to V_{\beta}$
such that the critical of $i$ is $\ka$, $\mathrm{rank}(x^*)<i(\ka)$,
and $i(d^*)(\ka)=x^*$''. 
\end{quote}

Now, $j \restriction V_{\mathrm{rank}(x^*)}$
is an elementary embedding from $V_{\mathrm{rank}(x^*)}$ to
$V_{j(\mathrm{rank}(x^*))}$ with critical point $\ka$,
$j(\ka)>\mathrm{rank}(x^*)$, and $j(d^*)(\ka)=x^*$.
Because $j(\mathrm{rank}(x^*))<\gamma$ and Fact \ref{GS thm},
in $V_\gamma$, $\col(\om, \size{V_{\mathrm{rank}(x^*)}})$ forces 
that ``there is  $\beta>\mathrm{rank}(x^*)$ and
an elementary embedding $i:V_{\mathrm{rank}(x^*)} \to V_{\beta}$
such that the critical of $i$ is $\ka$, $\mathrm{rank}(x^*)<i(\ka)$,
and $i(d^*)(\alpha)=x^*$''.
Because $V_{j(\ka)} \prec V_\gamma$,
the same thing holds in $V_{j(\ka)}$.
This is a contradiction.
\end{proof}

\begin{prop}\label{Laver diamond}
Suppose $\ka$ is virtually extendible.
Then there is a function $d:\ka \to V_\ka$ such that 
for every $\alpha>\ka$ and set $x \subseteq V_\alpha$,
$\col(\om, \size{V_\alpha})$ forces
that ``there is $\beta>\alpha$ and an elementary embedding
$j:V_\alpha \to V_\beta$ with critical point $\ka$,
$\alpha<j(\ka)$, and $j(d)(\ka)=x$''.
\end{prop}
\begin{proof}
Let $d^*$ be the function from Lemma \ref{weak laver diamond}.
Define the function $d$ on $\ka$ by $d(\alpha)=y$ if $d(\alpha)$ is of the form $\seq{y, \beta}$ for some $\beta \ge \om$,
and $d(\alpha)=0$ otherwise.
We check that $d$ is as required.

Fix $\alpha>\ka$ and $x \subseteq V_\alpha$.
By using a flat pairing function, we can assume that the rank of $\seq{x,\alpha}$ is $\alpha$.
By Lemma \ref{weak laver diamond}, 
$\col(\om, \size{V_{\alpha}})$ forces that
``there is $\beta$ and an elementary embedding $j:V_\alpha \to V_{\beta}$
with critical point $\ka$, $\gamma<j(\ka)$, and
$j(d^*)(\ka)=\seq{x,\alpha}$''.
By the definition of $d$,
$\col(\om, \size{V_{\alpha}})$ also forces that
``there is $\beta$ and an elementary embedding $j:V_{\alpha} \to V_{\beta}$
with critical point $\ka$, $\alpha<j(\ka)$, and
$j(d)(\ka)=x$''.
\end{proof}

We use the notion of \emph{subcomplete forcing}, which is introduced by Jensen.
We do not present the definition of subcompleteness in this paper, because 
it is complicated and we do not need it.
See Jensen \cite{J} for details.
Here we present several facts about subcomplete forcing from \cite{J}.
\begin{fact}\label{4.3}
\begin{enumerate}
\item Every subcomplete forcing does not add new reals, hence preserves $\om_1$.
\item Every $\sigma$-closed forcing is subcomplete.
\end{enumerate}
\end{fact}

\begin{fact}\label{4.4}
Let $\bbP$ be a subcomplete forcing notion, and $\dot \bbQ$ a $\bbP$-name for a
subcomplete forcing notion.
Then the forcing product $\bbP *\dot \bbQ$ is subcomplete.
\end{fact}

For a poset $\bbP$,
let $\delta(\bbP)$ be the minimum cardinality of dense subsets of $\bbP$.

\begin{fact}\label{4.5}
Let $\seq{\bbP_\alpha, \dot \bbQ_\alpha \mid \alpha<\beta}$ be 
a revised countable support iteration such that for every $\alpha<\beta$:
\begin{enumerate}
\item $\Vdash_{\bbP_\alpha}$``\,$\dot \bbQ_\alpha$ is subcomplete''.
\item $\Vdash_{\bbP_{\alpha+1}}$``\,$\delta(\bbP_\alpha) \le \om_1$''.
\end{enumerate}
Then $\bbP_\beta$ is subcomplete.
\end{fact}

Recall that \emph{Namba forcing}, denoted by $\mathrm{Nm}$,
is the set of all $T \subseteq {}^{<\om} \om_2$ such that:
\begin{enumerate}
\item For every $t \in T$ and $s \subseteq t$ we have $s \in T$.
\item There is $t \in T$ such that for every $s \in T$, $s \subseteq t$ or
$t \subseteq s$. Such a $t$ is called the \emph{stem} of $T$.
\item For every $s \in T$, if $s$ is an extension of the stem of $T$,
then the set $\{\alpha <\om_2 \mid s^\frown \seq{\alpha} \in T\}$ has cardinality $\om_2$.
\end{enumerate}
For $T_0, T_1 \in \mathrm{Nm}$, define $T_0 \le T_1$ if $T_0\subseteq T_1$.

\begin{fact}\label{4.7}
Suppose the Continuum Hypothesis (CH).
\begin{enumerate}
\item Namba forcing is subcomplete,
hence preserves $\om_1$.
\item $\mathrm{Nm}$ forces that $\cf(\om_2^V)=\om$.
\end{enumerate}
\end{fact}

We are ready to prove the theorem.
\begin{thm}\label{4.8}
Suppose $\ka$ is virtually extendible.
Then there is a forcing extension in which
$\om_1$ is preserved and $\ka=\om_2$ is generically extendible.
\end{thm}

\begin{proof}
Suppose $\ka$ is virtually extendible.
Let $d:\ka \to V_\ka$ be a function from Proposition \ref{Laver diamond}.
Let $R=\{\mu <\ka \mid \mu$ is inaccessible, $d``\mu \subseteq V_\mu$, $d(\mu)$ is an ordinal $\ge \mu\}$.
We define the $\ka$-stage revised countable support iteration $\seq{\bbP_\alpha, \dot \bbQ_\alpha \mid \alpha<\ka}$
as the following manner:
\begin{enumerate}
\item Suppose $\alpha \in R$ and $\bbP_\alpha$ forces that $\alpha=\om_2$.
Then $\dot \bbQ_\alpha$ is a $\bbP_\alpha$-name such that
\[
\Vdash_{\bbP_\alpha} \text{``} \dot{\bbQ}_\alpha=\dot{\col}(\alpha, d(\alpha)) * \dot{\mathrm{Nm}} * \dot{\col}(\om_1, \size{\bbP_\alpha}) \text{''}.
\]
\item Otherwise, $\dot \bbQ_\alpha$ is a $\bbP_\alpha$-name with
\[
\Vdash_{\bbP_\alpha} \text{``} \dot{\bbQ}_\alpha=\dot{\col}(\om_1, \size{\bbP_\alpha}) \text{''}.
\]

\end{enumerate}

By Facts \ref{4.3}, \ref{4.4}, \ref{4.5}, and \ref{4.7}, 
$\bbP_\ka$ is subcomplete, in particular preserves $\om_1$.
Moreover, since $\ka$ is inaccessible and $\size{\bbP_\alpha}<\ka$ for every $\alpha<\ka$,
one can check that $\bbP_\ka$ has the $\ka$-c.c.,
and $\bbP_\ka$ forces $\ka=\om_2$.
We shall show that $\bbP_\ka$ forces that ``$\ka$ is generically extendible''.

Take a $(V, \bbP_\ka)$-generic $G$.
To show that $\ka$ is generically extendible in $V[G]$,
fix an arbitrary large inaccessible $\theta>\ka$.
Since $\ka$ is virtually extendible in $V$ and $d$ is a Laver diamond,
there is some large $\gamma>\ka$
such that in $V^{\col(\om, \gamma)}$,
there are $\overline{\theta}>\theta$ and 
an elementary embedding $j:V_\theta \to V_{\overline{\theta}}$
with critical point $\ka$, $\theta<j(\ka)$, and $j(d)(\ka)=\theta$.
We may assume that $\overline{\theta}$ is inaccessible in $V$,
and $\gamma>{\overline{\theta}}$.
We know that $j(\bbP_\ka)=\bbP_{j(\ka)}$ which is of the form $\bbP_\ka *\dot \bbQ_\ka * \bbP_{tail}$.

In $V[G]$, since $\bbP_\ka$ forces that $\ka=\om_2$, $\ka \in j(R)$, and
$j(d)(\ka)=\theta$,
we know that $\bbQ_\ka$ is of the form
$\col(\ka, \theta)*\mathrm{Nm}*\col(\om_1, \size{\bbP_\ka})$.
Now take a $(V[G], 
\col(\ka, \theta)*\mathrm{Nm}*\col(\om_1, \size{\bbP_\ka}))$-generic
$H_0 *H_1*H_2 \in V^{\col(\om,\gamma)}$.
Take a surjection
$h_0 \in V[G][H_0]$ from $\ka$ to $\theta$.
We know $h_0 \restriction \alpha \in V[G]$ for every $\alpha<\ka$.
Since $\theta$ is inaccessible in $V[G]$,
there is a surjection from $\theta$ onto $V_\theta^{V[G]}$.
Hence in $V[G][H_0]$,
we can take a surjection $h:\ka \to V_\theta^{V[G]}$
such that $h \restriction \alpha \in V[G]$ for every $\alpha<\ka$.

In $V[G][H_0][H_1]$,
since $H_1$ is $\mathrm{Nm}$-generic and $\ka=\om_2$, 
there is  a cofinal map from $\om$ into $\ka$.
Fix an $\om$-cofinal sequence $\seq{\ka_n \mid n<\om}$ of $\ka$.
For $n<\om$, let $X_n=h``\ka_n$.
We know that $X_n \in V[G]$, $\size{X_n}^{V[G]} \le \om_1$,
and $\bigcup_n X_n=V_\theta^{V[G]}$.

In $V^{\col(\om, \gamma)}$, we can take
a $(V, \bbP_{j(\ka)})$-generic $G'$
which extends $G*H_0*H_1*H_2$.
We have $V_\theta[G] =V^{V[G]}_\theta$, $V_{\overline{\theta}}[G']=V^{V[G']}_{\overline{\theta}}$,
and we can extend $j$ to $j:V_\theta[G] \to V_{\overline{\theta}}[G']$ by the 
canonical way.
We note that this embedding $j$ is living in $V^{\col(\om,\gamma)}$ but may not  be in $V[G']$.

\begin{claim}
$j\restriction X_n \in V[G']$
for every $n<\om$.
\end{claim}
\begin{proof}
In $V[G]$, fix a bijection $\pi :\om_1 \to X_n$.
$\pi \in V_\theta^{V[G]}$, so we have $j(\pi) \in V_{\overline{\theta}}^{V[G']}$.
Then $j\restriction X_n$ is definable in $V[G']$
as $j(a)=j(\pi)(\pi^{-1}(a))$ for every $a \in X_n$.
\end{proof}

In $V[G']$,
let $T$ be the set of all partial elementary embeddings $i$ 
from $V^{V[G]}_\theta$ to $V^{V[G']}_{\overline{\theta}}$ 
such that:
\begin{enumerate}
\item $\dom(i)=X_n$ for some $n<\om$.
\item the critical point of $i$ is $\ka$,
and $i(\ka)=j(\ka)$.
\end{enumerate}
The set $T$ with the inclusion forms a tree of height $\om$.
Since $j \restriction X_n \in T$ for every $n<\om$,
the set $\{j \restriction X_n \mid n<\om\}$ is a cofinal branch of $T$ living in $V^{\col(\om, \gamma)}$.
So $T$ is ill-founded in $V^{\col(\om, \gamma)}$.
By the absoluteness of the ill-foundedness,
$T$ is also ill-founded in $V[G']$.
If $B \in V[G']$ is a cofinal branch of $T$,
then $B$ generates an elementary embedding $i:V^{V[G]}_\theta \to V^{V[G']}_{\overline{\theta}}$
with critical point $\ka$ and $i(\ka)=j(\ka)$.
Therefore, $V[G']$ is a generic extension of $V[G]$,
and in $V[G']$ we can take an elementary embedding 
from $V^{V[G]}_\theta$ to $V^{V[G']}_{\overline{\theta}}$, as required.
\end{proof}

By the previous theorem, $\om_2$ can be the least generically extendible cardinal.

\begin{cor}\label{3.10}
If $\ZFC$+``there exists a virtually extendible cardinal'' is consistent,
then so is 
$\ZFC$+``$\om_2$ is the least generically extendible cardinal''.
\end{cor}
\begin{proof}
Suppose $V=L$ and $\ka$ is the least virtually extendible cardinal.
By Theorem \ref{4.8}, we can take a generic extension $V[G]$ in which 
$\ka=\om_2$ is generically extendible.
Then $\om_1^{V[G]}$ cannot be generically extendible,
otherwise we have that $\om_1^{V[G]}$ is virtually extendible in $L$
by Theorem \ref{2.2},
which contradicts the minimality of $\ka$.
\end{proof}

The following question is suggested by the referee, but 
the author does not have an answer.
\begin{question}
Starting with two virtually extendible cardinals,
can we force that both $\om_1$ and $\om_2$ are generically extendible cardinals?
\end{question}

\section{Above $\om_2$}
In this section we consider the consistency strength of the existence of a generically extendible $>\om_2$.
First we prove the following lemma which may be a kind of folklore.

\begin{lemma}\label{sharp}
Suppose that there is an elementary embedding $j:L_\alpha \to L_\beta$
with $\om_2 \le \mathrm{crit}(j)<(\mathrm{crit}(j)^+)^L \le \alpha$.
Then $0^{\#}$ exists.
\end{lemma}
\begin{proof}
Suppose $0^{\#}$ does not exist. 
Let $\ka=\mathrm{crit}(j)$ and 
$U=\{X \in \p(\ka) \cap L_\alpha \mid \ka \in j(X)\}$.
$U$ is an $L$-ultrafilter over $\ka$.
We claim that $\mathrm{Ult}(L, U)$ is well-founded.
If this claim is verified, then we can conclude that $0^{\#}$ exists,
but this contradicts  the assumption.

Suppose to the contrary that $\mathrm{Ult}(L, U)$ is ill-founded,
so there are functions $f_n \in L$ ($n<\om$) on $\ka$
such that $\{\gamma<\ka \mid f_{n+1}(\gamma) \in f_n(\gamma) \} \in U$ for every $n$.
Since $0^{\#}$ does not exist,
by the Jensen's covering lemma there is $Y \in L$ such that
$\{f_n \mid n<\om\} \subseteq Y$ and $\size{Y} =\om_1$.
Note that $\size{Y}=\om_1<\om_2 \le \ka$,
hence $\size{Y}^L<\ka$.

In $L$, take a large limit $\gamma>\ka$ with $Y \in L_\gamma$,
and take $H \prec L_\gamma$ with $\ka \cup Y \subseteq H$ and $\size{H}^L = \ka$.
Let $N$ be the transitive collapse of $H$, and $\pi:H \to N$ be the collapsing map.
Take $\delta$ with $N=L_\delta$. Since $H \in L$, we know $\delta<(\ka^+)^L$.
For each $n \in \om$, let $\overline{f}_n=\pi(f_n)$.
Each $\overline{f}_n$ is a function on $\ka$
and $\{\gamma<\ka \mid \overline{f}_{n+1}(\gamma) \in \overline{f}_n(\gamma) \} \in U$.
Now, because $N=L_\delta$ and $\delta<(\ka^+)^L \le \alpha$,
we have $N \in L_\alpha$, and $\{\overline{f}_n \mid n<\om\} \subseteq L_\alpha$. 
By the definition of $U$,
we have $j(\overline{f}_{n+1})(\ka) \in j(\overline{f}_n)(\ka)$ for every $n<\om$,
this is a contradiction.
\end{proof}

Now Theorem \ref{thm2} follows from this lemma.
%
\begin{proof}[Proof of Theorem \ref{thm2}]
Let $\ka$ be a generically extendible cardinal $>\om_2$.
Let $\la=\ka^+$.
Take a poset $\bbP$ and a $(V, \bbP)$-generic $G$
such that, in $V[G]$,
there is $\gamma>\la$ and an elementary embedding 
$j:V_\la \to V_\beta^{V[G]}$
with critical point $\ka$ and $\la<j(\ka)$.
$j\restriction L_{\la}$ is an elementary embedding from $L_{\la}$ to $L_\beta$.
Since $\la >\ka>\om_2^V$,
we know $\om_1^{V[G]}=\om_1^N=\om_1^{V}$ and $\om_2^{V[G]}=\om_2^V \le \ka$.
Hence the assertion is immediate from the previous lemma.
\end{proof}

By using the Dodd-Jensen core model, we can increase the lower bound
of the consistency strength: If there is a generically extendible cardinal $>\om_2$,
then there is an inner model of a measurable cardinal.
However we do not know the exact consistency strength 
of it.
\begin{question}
What is the exact consistency strength of
the existence of a generically extendible cardinal $>\om_2$?
\end{question}
The following question also arises from Corollary  \ref{3.10}.
\begin{question}
Is it consistent that $\om_3$ is the least generically extendible cardinal?
How about $\om_4, \om_5,\dotsc$, and $\om_{\om+1}$?
\end{question}

\section{Indestructibility of generically extendible cardinals}

In this section, we prove indestructibility phenomenons
of generically extendible cardinals without any preparation forcing.

\begin{prop}\label{8.1}
Suppose $\ka$ is generically extendible.
Then the generic extendibility 
of $\ka$ is indestructible 
by $\ka$-c.c.~forcing. 
\end{prop}
\begin{proof}
Let $\bbP$ be a $\ka$-c.c.\ poset.
Fix an arbitrary large limit ordinal $\alpha>\ka$ such that
$\bbP \in V_\alpha$.
Take a poset $\bbQ$ and a $(V,\bbQ)$-generic $H$
such that, in $V[H]$,
there is $\beta>\alpha$ and an elementary embedding $j:V_\alpha \to V_\beta^{V[H]}$ with
critical point $\ka$ and $\alpha<j(\ka)$.
Then take a $(V[H], j(\bbP))$-generic $G'$.
\begin{claim}
If $A \subseteq \bbP$ is a maximal  antichain with $A \in V$,
then $j``A \cap G' \neq \emptyset$.
\end{claim}
\begin{proof}[Proof of Claim]
Since $\bbP$ has the $\ka$-c.c. in $V$, we have $\size{A}<\ka$.
Then $j(A)=j``A \subseteq j(\bbP) \in V_\beta^{V[H]}$ is a maximal antichain with $j``A \in V[H]$.
Hence $j``A \cap G' \neq \emptyset$.
\end{proof}

Let $G=\{p \in \bbP \mid j(p) \in G'\}$.
By the previous claim, we have that $G \cap A \neq \emptyset$ for every maximal antichain $A \subseteq \bbP$ with $A \in V$.
Moreover one can check that $G$ is directed, that is, for every 
$p, q \in G$ there is $r \in G$ with $r \le p,q$.
%
Therefore $G$ is $(V, \bbP)$-generic, and $j``G \subseteq G'$.
Then $j:V_\alpha \to V_\beta^{V[H]}$ can be extended to
$j:V_\alpha^{V[G]} \to V_\beta^{V[H][G']}$.
Hence, in $V[H][G']$ which is a generic extension of $V[G]$,
 we can take a generic elementary embedding from $V_\alpha^{V[G]}$ to $V_\beta^{V[H][G']}$.
This argument shows that $\bbP$ preserves the  generic extendibility of $\ka$.
\end{proof}

When $\ka=\om_1$, we can strengthen the previous proposition as follows.
\begin{prop}\label{proper indes}
Suppose $\om_1$ is generically extendible.
Then the generic extendibility 
of $\om_1$ is indestructible by  proper forcing. 
\end{prop}
\begin{proof}
Let $\bbP$ be a proper poset.
Fix a large regular cardinal $\la$,
and a large limit ordinal $\alpha>\la$ with $\bbP \in V_\alpha$.
Take a poset $\bbQ$ and a $(V,\bbQ)$-generic $H$
such that, in $V[H]$,
there is $\beta$ and an elementary embedding $j:V_{\alpha} \to V_{\beta}^{V[H]}$ with
critical point $\om_1^V$ and $\alpha<j(\om_1^V)$.
Since $\alpha>\la$ is large, we have that $j(\la)$ is regular and 
$j(\bbP)$ is proper in $V[H]$.
We know $j``H_\la^V \prec j(H_\la^V)=H_{j(\la)}^{V[H]}$.
Since $j(\bbP)$ is proper and $j``H_\la^V$ is countable in $V[H]$,
there is a generic condition $p^* \in j(\bbP)$ for $j``H_\la^V$.
Then take a $(V[H], j(\bbP))$-generic filter $G^*$ with $p^* \in G^*$.
Let $G=\{p \in \bbP \mid j(p) \in G^*\}$.
We show that $G$ is $(V, \bbP)$-generic;
First, take a dense set $D \in V$ in $\bbP$.
Then $j(D) \in j``H_\theta^V$. Since $p^*$ is a generic condition,
$j``D =j(D) \cap j``H_\theta^V$ is predense below $p^*$.
$G^*$ contains $p^*$, hence there is $q \in \bbP$ with
$j(q) \in j``D \cap G^*$, hence $q \in D \cap G$.
By the elementarity of $j$, for each $p,q \in G$,
since $j(p), j(q) \in G'$, we know that $p$ and $q$ are compatible.
Because for every dense set $D$ in $V$, $G^*$ meets $j(D)$,
it is routine to check that $p$ and $q$ have a lower bound in $G$.
By the definition of $G$, we can extend $j$ to $j:V_\alpha^{V[G]} \to V_\beta^{V[H][G^*]}$.
This completes the proof.
\end{proof}

As stated in the introduction,
if there is a proper class of Woodin cardinals,
then every regular uncountable cardinal is generically extendible.
This means that, under the existence of proper class many Woodin cardinals,
the generic extendibility of each regular uncountable cardinal is indestructible by
forcing which preserves its regularity.
This observation suggests the question: Does the generic extendibility of $\ka$ is preserved by
forcing which preserves the regularity of $\ka$?
An answer is negative. In the next section, 
we will show that it is consistent that
$\om_1$ is generically extendible but
there is an $\om_1$-preserving forcing which destroys the generic extendibility of $\om_1$.

On the other hand, the following question arises from 
Laver's indestructibility of supercompactness by directed closed forcing, and 
the destructibility of an extendible cardinal by closed forcing.
\begin{question}
Let $\ka$ be a generically extendible cardinal.
Does every $\mathop{<}\ka$-closed (or $\mathop{<}\ka$-directed closed) forcing
preserve the generic extendibility of $\ka$?
\end{question}

We know only the following partial answer.
In the resulting model $V$ of Theorem \ref{4.8},
the generic extendibility of $\om_2$ is preserved by $\sigma$-closed $<\om_2$-Baire forcing.
Let us sketch the proof.
Fix a $\sigma$-closed $<\om_2$-Baire poset $\bbQ$.
Take a large $\alpha>\om_2$.
In $V$,
we can find a poset $\bbP$ which forces that:
\begin{enumerate}
\item There are $\beta$ and an elementary embedding $j:V_\alpha \to V_\beta^{V^\bbP}$ such that
the critical point $\ka$ and $\alpha<j(\om_2^V)$.
\item There is a family $\{X_n \mid n<\om\} \in V^{\bbP}$ with
$X_n \in V$, $\size{X_n}^V \le \om_1^V$, and $V_\alpha=\bigcup_n X_n$.
\end{enumerate}
Since $\bbQ$ is $<\om_2$-Baire in $V$ and $\size{X_n}^V \le \om_1^V$,
in $V^\bbP$ we can find a descending sequence $\{q_n \mid n<\om\}$ in $\bbQ$
such that $q_n \in \bigcap\{D \in X_n \mid D$ is dense in $\bbQ\}$.
Because $j(\bbQ)$ is $\sigma$-closed in $V_\beta^{V^\bbP}$,
we can find a lower bound $q^* \in j(\bbQ)$ of $\{j(q_n) \mid n<\om\}$.
If $H^*$ is $(V^\bbP, j(\bbQ))$-generic with $q^* \in H^*$,
then $H=\{q \in \bbQ \mid j(q) \in H^*\}$ is $(V, \bbQ)$-generic,
and $j$ can be extended to $j:V_\alpha^{V[H]} \to V_\beta^{V^\bbP[H^*]}$.

\section{Generic absoluteness}
In this section, we study ${\bf \Sigma}_n^1$-absoluteness.
Recall that,
for a class $\Gamma$ of posets and 
$n<\om$, let us say that \emph{${\bf \Sigma}_n^1$-absoluteness for $\Gamma$ holds}
if for every real $r$ and a $\Sigma_n^1$ or $\Pi_n^1$-formula $\varphi(x)$ of the reals,
if $\varphi(r)$ holds in $V$ then so does in $V^{\bbP}$ for every $\bbP \in \Gamma$.
If $\Gamma$ is the class of all posets,
then we call it just \emph{${\bf \Sigma}_n^1$-absoluteness}.

\begin{prop}\label{9.1}
Suppose $\om_1$ is generically extendible.
Then ${\bf \Sigma}^1_3$-absoluteness holds.
\end{prop}
\begin{proof}
By the Shoenfiled absoluteness,
it is enough to show that for every poset $\bbP$, $(V,\bbP)$-generic $G$, 
$r \in \bbR$, and a  $\Sigma^1_3$-formula $\exists y \varphi(x,y)$ (where $\varphi$ is $\Pi^1_2$),
if $\exists y \varphi(r,y)$ holds in $V[G]$ then 
so does in $V$.

Fix a large $\alpha$ with $\bbP \in V_\alpha$,
and take a poset $\bbQ$ such that,
in $V^\bbQ$, there is an elementary embedding $j:V_\alpha \to V_\beta^{V^\bbQ}$ with critical point $\om_1^V$
and $j(\om_1^V)>\alpha$.
Since $\p(\bbP)^V$ is countable in $V^{\bbQ}$,
we may assume that $G \in V^\bbQ$.
Hence $V^\bbQ$ is a generic extension of $V[G]$.
In $V[G]$, since $\exists y \varphi(r,y)$ holds,
there is $s \in \bbR^{V[G]}$ such that $\varphi(r,s)$ holds in $V[G]$.
$\varphi(r,s)$ is $\Pi^1_2$, and by the Shoenfield absoluteness,
$\varphi(r,s)$ still holds in $V^\bbQ$.
Hence 
 $\exists y \varphi(r,y)$ holds in $V_\beta^{V^\bbQ}$.
By the elementarity of $j$,
we have that $\exists y \varphi(r,y)$ holds in $V$.
\end{proof}

\begin{prop}\label{9.2+}
Suppose $\om_1$ is generically extendible.
Let $\Gamma$ be a class of $\om_1$-preserving forcing notions such that
every $\bbP \in \Gamma$ preserves the generic extendibility of $\om_1$.
Then ${\bf \Sigma}^1_4$-absoluteness for $\Gamma$ holds.
\end{prop}
\begin{proof}
Again, it is enough to show that for every poset $\bbP \in \Gamma$, $(V,\bbP)$-generic $G$, 
$r \in \bbR$, and a  $\Sigma^1_4$-formula $\exists y \varphi(x,y)$ (where $\varphi$ is $\Pi^1_3$),
if $\exists y \varphi(r,y)$ holds in $V[G]$ then 
so does in $V$.

In $V[G]$, take $s \in \bbR$ such that $\varphi(r,s)$ holds.
By the assumption, $\om_1=\om_1^V$ is generically extendible in $V[G]$.
Hence by Proposition \ref{9.1},
$\varphi(r,s)$ holds in any generic extension of $V[G]$.
In $V$, fix a large $\alpha$,
and take a poset $\bbQ$ such that,
in $V^\bbQ$, there is an elementary embedding $j:V_\alpha \to V_\beta^{V^\bbQ}$ with critical point $\om_1^V$
and $j(\om_1^V)>\alpha$.
We may assume 
that 
$V^\bbQ$ is a generic extension of $V[G]$.
Thus $\varphi(r,s)$ holds in $V^\bbQ$,
and by the elementarity of $j$,
$\exists y \varphi(r,y)$ holds in $V$.
\end{proof}

Since every proper forcing preserves the
generic extendibility of $\om_1$
by Proposition \ref{proper indes}, we have:
\begin{cor}
Suppose $\om_1$ is generically extendible.
Then ${\bf \Sigma}^1_4$-absolute-ness for proper forcing holds.
\end{cor}

On the other hand,
it is possible that ${\bf \Sigma}^1_4$-absoluteness for $\om_1$-preserving forcing does not hold.
\begin{fact}[Remark after Theorem 8 in Bagaria-Friedman \cite{BF}]\label{BF}
If ${\bf \Sigma}^1_4$-absoluteness for $\om_1$-preserving forcing holds,
then every set has a $\#$.
\end{fact}
\begin{prop}
It is consistent that 
$\om_1$ is generically extendible but 
${\bf \Sigma}^1_4$-absoluteness for $\om_1$-preserving forcing does not hold.
\end{prop}
\begin{proof}
By Theorem \ref{2.5}, we can construct a model in which $\om_1$ is generically extendible but $0^{\#}$ does not exist.
This is a required model.
\end{proof}
Now combining Fact \ref{BF} with Proposition \ref{9.2+},
we have the destructibility of the generic extendibility of $\om_1$.
\begin{cor}
\begin{enumerate}
\item Suppose every $\om_1$-preserving forcing preserves the
generic extendibility of $\om_1$.
The every set has a $\#$.
\item It is consistent that
$\om_1$ is generically extendible,
and there is an $\om_1$-preserving forcing
which destroys the 
generic extendibility of $\om_1$.
\end{enumerate}
\end{cor}
\begin{proof}
(1). By Proposition \ref{9.2+},
if every $\om_1$-preserving forcing preserves the
generic extendibility of $\om_1$,
then ${\bf \Sigma}^1_4$-absoluteness for $\om_1$-preserving forcing holds.
The conclusion follows from Fact \ref{BF}.

(2). Suppose $\om_1$ is generically extendible but $0^{\#}$ does not exist.
By (1), there is an $\om_1$-preserving forcing which destroys the 
generic extendibility of $\om_1$.
\end{proof}

\section{Boolean valued second order logic}

In this section, we study some applications of generically extendible cardinals to Boolean valued second order logic.
See Ikegami-V\"a\"an\"anen \cite{IV} and Ikegami \cite{IV2} for Boolean valued second order logic.
Here we recall some definitions and facts about it.

A sentence of Boolean valued second order logic is 
a sentence of second order logic of a relational language $\{R_0,\dotsc, R_n\}$,
where each $R_i$ is an $m_i$-ary relation symbol.
A structure of Boolean valued second order logic
is of a form $M=(A, \mathbb{B}, \{R_0^M,\dotsc, R_n^M\})$
where: 
\begin{enumerate}
\item $A$ is a non-empty set.
\item $\mathbb{B}$ is a complete Boolean algebra.
\item $R_i^M$ is a function from ${}^{m_i} A$ to $\mathbb B$.
\end{enumerate}
Each $R_i^M$ can be seen as a $\mathbb B$-name,
and whenever $G$ is $(V,\mathbb B)$-generic,
the interpretation $R_i^M/G$ of $R_i^M$ by $G$
is an $m_i$-ary relation on $A$.
Let $M/G=(A,\{R_i^M/G\})$, this is a relational structure.

For a
structure of Boolean valued second order logic
$M=(A, \mathbb{B}, \{R_i^M\})$
and a sentence $\varphi$,
let us say that 
$M$ is a \emph{model of $\varphi$}
if $\mathbb B$ forces that
``$M/\dot G$ is a model of $\varphi$ with respect to the full semantics'',
where $\dot G$ is a $\mathbb B$-name for a canonical generic filter\footnote{
This definition is different from the original one in \cite{IV}, but is equivalent. See Lemma 2.4 in \cite{IV}.}.

Let us consider some cardinal numbers defined by strong logic.
The compactness number of the full second order logic $\ka^2$
is the least
cardinal $\ka$ such that for every theory $T$ of second order logic,
if every subtheory $T' \subseteq T$ with cardinality $<\ka$ has a model
(with respect to the full semantics), then $T$ has a model.
Ikegami \cite{IV2}  defined
the compactness number of Boolean valued second order logic $\ka^{2b}$ as an analog of $\ka^2$:
$\ka^{2b}$ is the least cardinal $\ka$ such that for every theory $T$ of  Boolean valued second order logic,
if every subtheory $T' \subseteq T$ with cardinality $<\ka$ has a model,
then $T$ has a model.

Magidor \cite{Magidor} proved that
$\ka^2$ is equal to the least extendible cardinal,
and Ikegami proved an analog result
for $\ka^{2b}$:
\begin{fact}[Ikegami \cite{IV2}]
$\ka^{2b}$  is equal to the least generically extendible cardinal.
\end{fact}

The Hanf number of the second order logic $h^2$ 
is the least cardinal $\ka$
such that for every sentence $\varphi$ of second order logic,
if $\varphi$ has a model (with respect to the full semantics) of cardinality $\ge \ka$,
then $\varphi$ has  models with arbitrary large cardinalities.
It is known that $h^{2}$ is greater than the first fix point of the $\beth$-function,
and if there is an extendible cardinal,
then $h^2$ is strictly smaller than the least extendible cardinal (\cite{Magidor}).

We define the cardinality of a Boolean valued second order structure $(A, \mathbb B, \{R_i^M\})$  as the 
cardinality of $A$.
The Hanf number of  Boolean valued second order logic $h^{2b}$ is defined  as the least cardinal $\ka$
such that for every sentence $\varphi$ of Boolean valued second order logic,
if $\varphi$ has a model of cardinality $\ge \ka$,
then $\varphi$ has  models with arbitrary large cardinalities.

In Ikegami-V\"a\"an\"anen \cite{IV}, they showed that, under a certain assumption, 
a supercompact cardinal is an upper bound of $h^{2b}$.
We prove that a generically extendible cardinal is an upper bound of
$h^{2b}$, and $h^{2b}$ can be small as $\om$ or $\om_1$.

\begin{lemma}\label{basic boolean model}
Let $\varphi$ be a sentence of Boolean valued second order logic,
and $\bbP$  a poset.
Let $\alpha$ be an ordinal,
and suppose $\bbP$  forces that ``there is a model $M=(\alpha, \mathbb B, \{R_i^M\})$ of $\varphi$''.
Then in $V$, there is a model $N$ of $\varphi$ which is of the form 
$(\alpha, \mathbb C, \{R^N_i\})$.
\end{lemma}
\begin{proof}
Fix  $\bbP$-names $\dot{\mathbb B}$ and  $\{\dot R_i^M\}$
such that $\Vdash_\bbP$``$M=(\alpha, \dot{\mathbb B}, \{\dot R_i^M\})$ is a model of $\varphi$''.
Let $\bbC$ be the completion of the forcing product $\bbP* \dot{\mathbb{B}}$,
and define the structure $N=(\alpha, \bbC, \{R_i^N\})$
by: For each $\vec{x} \in {}^{m_i} \alpha$,
\[
R_i^N(\vec{x}) = \text{ the Boolean value of the statement ``$\vec{x} \in \dot R_i^M/\dot H$'',}
\]
where $\dot H$ is a $\bbC$-name for a canonical $(V^\bbP, \dot{\mathbb B})$-generic filter.

It is straightforward to show that
$N$ is a model of $\varphi$.
\end{proof}
\begin{prop}\label{9.3}
If $\ka$ is generically extendible,
then $h^{2b} \le \ka$.
\end{prop}
\begin{proof}
Take a sentence $\varphi$, and suppose $\varphi$ has a model
$M=(A, \mathbb{B}, \{R_i^M\})$ of cardinality $\la \ge \ka$.
We may assume that $A$ is $\la$.
Fix an arbitrary large $\alpha$,
and take a poset $\bbP$ which forces
that there is $\beta$ and an elementary embedding $j:V_\alpha \to V_\beta^{V^\bbP}$
with critical point $\ka$ and $\alpha<j(\ka)$.

Consider $j(M)=(j(\la), j(\mathbb{B}), \{j(R_i^M)\})$, which is a model of $\varphi$ in $V_\beta^{V^\bbP}$.
Since $V_\beta^{V^\bbP}$ is a rank initial segment of $V^\bbP$,
$j(M)$ is in fact a model of $\varphi$ in $V^\bbP$ with cardinality $j(\la)>\alpha$.
By Lemma \ref{basic boolean model},
in $V$, we can find a model $N$ of $\varphi$ with cardinality $j(\la) > \alpha$.
\end{proof}

\begin{prop}\label{9.4}
If  $\ka$ is generically extendible,
then $h^{2b} <\ka$.
\end{prop}
\begin{proof}

For a sentence $\varphi$ of Boolean valued second order logic,
let $\mathrm{sp}(\varphi)=\{\la \mid \varphi$ has a model of cardinality $\la\}$.

First suppose $\ka$ is a successor cardinal, say $\ka=\mu^+$.
Let $\varphi$ be a sentence of Boolean valued second order logic.
It is enough to show that if $\mu \in \mathrm{sp}(\varphi)$
then $\mathrm{sp}(\varphi)$ is unbounded in the ordinals.
Take a model $M=(\mu, \mathbb B, \{R^M_i\})$ of $\varphi$.

Fix a large $\alpha>\ka$,
and take a poset $\bbP$ which forces that:
There is $\beta$ and an elementary embedding $j:V_\alpha \to V_\beta^{V^\bbP}$.
Then $j(M)$ is of the form $(\mu, j(\mathbb{B}), \{j(R^M_i)\})$, and in $V^{\bbP}$ it is a model of $\varphi$.
In $V^{\bbP}$, $\mu$ is a cardinal and $\size{\alpha}=\mu$.
Then it is easy to take
a model of $\varphi$ which is of the form $N=(\alpha, j(\mathbb B), \{R^N_i\})$.
By Lemma \ref{basic boolean model}, in $V$, we can take a model of $\varphi$ of cardinality $\alpha$.

Next suppose $\ka$ is a limit cardinal.
Since there are countably many sentences,
it is enough to see that
if $\sup(\mathrm{sp}(\varphi) \cap \ka)=\ka$,
then $\mathrm{sp}(\varphi)$ is unbounded in the ordinals.

Fix a large $\alpha>\ka$,
and take a poset $\bbP$ which forces that:
There is $\beta$ and an elementary embedding $j:V_\alpha \to V_\beta^{V^\bbP}$.
By the elementarity,
$j(\mathrm{sp}(\varphi) \cap \ka)$ is unbounded in $j(\ka)$,
hence we can take a model $M=(\gamma, \mathbb B, \{R_i^M\})$ of $\varphi$ with 
$\gamma>\alpha$.
By Lemma \ref{basic boolean model} again,
in $V$, we can take a model of $\varphi$ with cardinality $\gamma>\alpha$.
\end{proof}


\begin{cor}
If $\om_1$ is generically extendible, 
then $h^{2b}=\om$. 
\end{cor}

Now we point out that the statement $h^{2b}=\om$  does not have  a large cardinal strength.

\begin{prop}\label{h2b ctble}
There is a forcing extension in which $h^{2b}=\om$ holds.
\end{prop}
\begin{proof}
Fix  a large $n<\om$ and $\ka$ with $V_\ka \prec_{\Sigma_n} V$.
Take a $(V,\col(\om, \ka))$-generic $G$.
We shall show that $h^{2b}=\om$ in $V[G]$.

In $V[G]$, take a sentence $\varphi$ and 
an infinite model $M=(A, \mathbb{B}, \{R^M_i\})$ of $\varphi$.
Since $\ka$ is countable in $V[G]$, we may assume that $A$ is an ordinal $\alpha>\ka$.
Then the following claim follows from $\Sigma_n$-elementarity of $V_\ka$.
\begin{claim}
In $V$, there is an arbitrary large $\la$
such that $\col(\om, \la)$ forces that
``there is $\beta>\la$ and a model $N=(\beta, \mathbb{C}, \{R^N_i\})$ of $\varphi$''.
\end{claim}
In $V$, fix such a $\la$ with $\la>\alpha$.
Then we can take a $(V, \col(\om, \la))$-generic $H$ with $G \in V[H]$,
so $V[H]$ is a generic extension of $V[G]$.
In $V[H]$, there is a model $\varphi$ which is of the form
$N=(\beta, \mathbb{C}, \{R^N_i\})$ for some $\beta>\la$.
By Lemma \ref{basic boolean model},
in $V[G]$ we can find a model
$\tilde{N}=(\beta, \tilde{\mathbb{C}}, \{R^{\tilde{N}}_i\})$ of $\varphi$.
\end{proof}

The existence of a generically extendible cardinal  has a large cardinal strength,
hence we have:
\begin{cor}
It is consistent that $h^{2b}=\om$ and there is no generically extendible cardinal.
\end{cor}

Next we show the consistency of $h^{2b}=\om_1$.
\begin{lemma}\label{9.2}
If $\om_1^L=\om_1$,
then $h^{2b} \ge \om_1$.
\end{lemma}
\begin{proof}
There is a sentence $\varphi$
such that for every structure $M$ of the form $(A, \mathbb B, \{R^M_i\})$,
$M$ is a model of $\varphi$ 
if and only if
$\mathbb B$ forces ``every countable ordinal is countable in $L$,
and $A$ is countable''.
Clearly $\varphi$ has a countable model
with cardinality $<\om_1^L$,
but has no model with cardinality $\ge \om_1^L$.
\end{proof}

\begin{prop}
Suppose $V=L$.
There is a forcing extension in which $h^{2b}=\om_1$ holds.
\end{prop}
\begin{proof}
As in the proof of Proposition \ref{h2b ctble}, fix  a large $n<\om$ and $\ka$ with $V_\ka \prec_{\Sigma_n} V$.
In  this case, take a $(V,\col(\om_1, \ka))$-generic $G$.
$\om_1^{V[G]}=\om_1^L$, hence $h^{2b} >\om$ holds in $V[G]$.
To show that $h^{2b} \le \om_1$, take a sentence $\varphi$ and 
an uncountable model $M=(\alpha, \mathbb{B}, \{R^M_i\})$ where $\alpha>\ka$.

By the $\Sigma_n$-elementarity of $V_\ka$,
in $V$, there is $\la>\alpha$
such that $\col(\om_1, \la)$ forces that
``there is $\beta>\la$ and a model $N=(\beta, \mathbb{C}, \{R^N_i\})$ of $\varphi$''.
As before, in $V[G]$, we can find a model of $\varphi$ with cardinality $\beta$.
\end{proof}

A similar argument shows the consistency of $h^{2b}=\om_n$ for $n \in \om$.

Finally we prove that if $\om_1$ is generically extendible, then 
the statement $h^{2b}=\om$ is preserved by forcing.
\begin{prop}\label{7.7}
Suppose $\om_1$ is generically extendible.
Then for every generic extension $V[G]$ of $V$,
$h^{2b}=\om$ holds in $V[G]$.
\end{prop}
\begin{proof}
Take a poset $\bbP$ and a $(V, \bbP)$-generic $G$.
In $V[G]$, fix a sentence $\varphi$ and a
model $M=(\la, \mathbb B, \{R^M_i\})$ of $\varphi$,
where $\la$ is an infinite cardinal.
Fix a large limit $\alpha>\la$.
We shall find a model of $\varphi$ with 
cardinality $ \ge \alpha$.

In $V$, by Lemma \ref{basic boolean model}, 
we can take a model $N=(\la, \mathbb C, \{R^N_i\})$ of $\varphi$. 
Take a poset $\bbQ$ and a $(V, \bbQ)$-generic $H$
such that in $V[H]$,
there is an elementary embedding $j:V_\alpha \to V_\beta^{V[H]}$ with
critical point $\om_1^V$ and $\alpha<j(\om_1^V)$. 
We may assume that $G \in V[H]$,
hence $V[H]$ is a generic extension of $V[G]$.
Consider the structure $j(N)=(j(\la), j(\mathbb{C}), \{j(R^N_i)\})$,
which is a model of $\varphi$ in $V[H]$.

First suppose that $\la$ is an uncountable cardinal.
By applying Lemma \ref{basic boolean model} between  $V[G]$ and $V[H]$,
we can find a model of $\varphi$ in $V[G]$ which is of the form $N'=(j(\la), \mathbb C', \{R^{N'}_i\} )$.
Since $j(\la)>j(\om_1) >\alpha$, $N'$ is a required large model of $\varphi$ in $V[G]$.

Next suppose $\la=\om$.
Then $j(N)$ is of the form $(\om, j(\mathbb{C}), \{j(R^N_i)\})$.
Since $\alpha$ is countable in $V[H]$,
we can take a model $\varphi$ which is of the form $\tilde{N}=(\alpha, \tilde{\mathbb C}, \{R_i^{\tilde{N}}\})$.
As in the uncountable case,
there is a model $\varphi$ of cardinality $\alpha$ in $V[G]$.
\end{proof}


\begin{cor}
It is consistent that $h^{2b}=\om$ and $\om_2$ is the least generically extendible cardinal.
\end{cor}
\begin{proof}
Suppose $V=L$, and let $\ka$ be a virtually extendible cardinal,
and $\la$ the minimal virtually extendible cardinal $>\ka$.
First collapse $\ka$ to $\om_1$.
By Proposition \ref{2.5},
$\ka=\om_1$ is  generically extendible in the extension.
In the extension, $\la$ remains  virtually extendible.
Next, collapse $\ka$ to be countable by the standard way.
In the second extension, $h^{2b}=\om$ 
by  Proposition \ref{7.7}, and $\la$ remains virtually extendible.
Finally, by Theorem \ref{4.8}, we can collapse $\la$ to  $\om_2$ so that 
$\la=\om_2$ is generically extendible in the extension.
In this third extension, 
$\om_2$ is generically extendible, $h^{2b}=\om$
by Proposition \ref{7.7}, and
$\om_1$ is not generically extendible.
\end{proof}

\begin{question}
Is it consistent that $h^{2b}=\om$ in $V$ but $h^{2b}>\om$ in some generic extension? 
\end{question}

\section*{Acknowledgments}
The author would like to thank Daisuke Ikegami,
as his questions led to the start of this research.
The author also thanks the referee for many valuable comments and useful suggestions.


\begin{thebibliography}{100}
\bibitem{BF} J.~Bagaria, S.~D.~Friedman,  
\emph{Generic absoluteness.}
Ann. Pure Appl. Logic 108, No. 1--3, 3--13 (2001). 
\bibitem{BGS}
J.~Bagaria, V.~Gitman, R.~Schindler,
\emph{Generic Vop\v enka's Principle, remarkable
cardinals, and the weak Proper Forcing Axiom}.
Arch. Math. Logic 56 (2017), no. 1-2, 1--20. 
\bibitem{BHTU} J.~Bagaria, J.~D.~Hamkins, K.~Tsaprounis, T.~Usuba,
\emph{Superstrong and other large cardinals are never Laver indestructible.}
Arch. Math. Logic 55, No. 1--2, 19--35 (2016). 
\bibitem{GS} V.~Gitman, R.~Schindler, 
\emph{Virtual large cardinals.}
Ann. Pure Appl. Logic 169, No. 12, 1317--1334 (2018). 
\bibitem{H2} J.~D.~Hamkins,
\emph{Small forcing makes any cardinal superdestructible.}
J. Symb. Log. 63, No. 1, 51--58 (1998). 
\bibitem{IV2} D.~Ikegami,  
\emph{Boolean-valued second-order logic revisited}, in preparation.
\bibitem{IV} D.~Ikegami, J.~V\"a\"an\"anen,
\emph{Boolean-valued second-order logic.}
Notre Dame J. Formal Logic 56, No. 1, 167--190 (2015). 
\bibitem{J} R.~Jensen, \emph{Subcomplete forcing and $\mathcal{L}$-forcing.} In:
Lecture Notes Series. Institute for Mathematical Sciences. National University of Singapore 27, 83--182 (2014). 
\bibitem{L} P.~B.~Larson, \emph{The stationary tower. Notes on a course by W. Hugh Woodin.}
University Lecture Series 32 (2004). 
\bibitem{Laver} R.~Laver, \emph{Making the supercompactness of $\nu$ indestructible under $\nu$-directed closed forcing.} Isr. J. Math. 29, 385-388 (1978). 
\bibitem{Magidor} M.~Magidor, \emph{On the role of supercompact and extendible cardinals in logic.}
Isr. J. Math. 10, 147--157 (1971). 
\end{thebibliography}
\end{document}